\documentclass[multi]{cambridge7A}
\usepackage[UKenglish]{babel}
\usepackage[longnamesfirst,sectionbib]{natbib}
\usepackage{chapterbib}
\usepackage{amsmath,amssymb,amsthm,amsfonts}
\usepackage{graphicx,epsfig,psfrag}
\usepackage{pict2e,color,calc,xspace,latexsym,ifthen,enumerate,url}
\usepackage{subfigure}
\setcitestyle{numbers,square,comma}
\allowdisplaybreaks[1]

\theoremstyle{plain}
\newtheorem{theorem}{Theorem}[section]
\newtheorem{lemma}[theorem]{Lemma}

\theoremstyle{definition}

\newcommand{\beq}{\begin{equation}}
\newcommand{\eeq}{\end{equation}}
\newcommand{\beqas}{\begin{eqnarray*}}
\newcommand{\eeqas}{\end{eqnarray*}}

\newcommand{\ve}{v^\varepsilon}

\def\PP{\mathbb{P}}
\def\ve{{\varepsilon}}

\def\le{\leqslant}

\def\ge{\geqslant}

\def\E{{\mathbb E}}
\def\R{{\mathbb R}}
\def\N{{\mathbb N}}
\def\T{{\mathbb T}}
\def\Z{{\mathbb Z}}

\def\a{{\alpha}}
\def\b{{\beta}}
\def\d{{\delta}}
\def\D{{\Delta}}
\def\g{{\gamma}}

\def\s{{\sigma}}
\def\l{{\lambda}}

\def\th{{\theta}}
\def\o{{\omega}}
\def\bc{\begin{color}{blue}}
\def\bg{\begin{color}{blue}}
\def\ec{\end{color}}
\def\eg{\end{color}}

\def\cF{{\cal F}}
\def\cL{{\cal L}}

\def\q{\quad}
\def\pd{{\partial}}

\def\capp{\operatorname{cap}}
\def\divv{\operatorname{div}}
\def\var{\operatorname{var}}
\def\rad{\operatorname{rad}}
\def\<{\langle}
\def\>{\rangle}

\def\sse{\subseteq}
\def\es{\emptyset}

\providecommand*\Index[1]{#1\index{#1}}
\providecommand*\undex[1]{} 

\begin{document}
  \alphafootnotes
  \author[P. H. Haynes, V. H. Hoang, J. R. Norris and K. C. Zygalakis]%
    {P. H. Haynes\footnotemark, V. H. Hoang\footnotemark,
      J. R. Norris\footnotemark\ and K. C. Zygalakis\footnotemark }

  \chapter[Homogenization for advection-diffusion in a perforated domain]{Homogenization for advection-diffusion in a perforated domain}

  \footnotetext[1]{Department of Applied Mathematics and Theoretical Physics,
Centre for Mathematical Sciences, Wilberforce Road, Cambridge CB3 0WA;
phh1@cam.ac.uk}
\footnotetext[2]{Division of Mathematical Sciences, Nanyang Technological
  University, SPMS--MAS--04--19, 21 Nanyang Link, Singapore 637371;
  VHHOANG@ntu.edu.sg}
\footnotetext[3]{Statistical Laboratory, Centre for Mathematical Sciences,
  Wilberforce Road, Cambridge, CB3 0WB;
  j.r.norris@statslab.cam.ac.uk}
\footnotetext[4]{Mathematical Institute, 24--29 St Giles', Oxford OX1 3LB;
  zygalakis@maths.ox.ac.uk. Supported by a David Crighton Fellowship and by
  Award No. KUK--C1--013--04 made by King Abdullah University of Science and
  Technology (KAUST).}
  \arabicfootnotes

  \contributor{Peter H. Haynes
    \affiliation{University of Cambridge}}

  \contributor{Viet Ha Hoang
    \affiliation{Nanyang Techological University, Singapore}}

  \contributor{James R. Norris
    \affiliation{University of Cambridge}}

  \contributor{Konstantinos C. Zygalakis
    \affiliation{University of Oxford}}
\renewcommand\thesection{\arabic{section}}
\numberwithin{equation}{section}
\renewcommand\theequation{\thesection.\arabic{equation}}
\numberwithin{figure}{section}
\renewcommand\thefigure{\thesection.\arabic{figure}}

 \begin{abstract}
The volume of a Wiener sausage constructed from a diffusion process with periodic, mean-zero, divergence-free velocity field, 
in dimension $3$ or more, is shown to have a non-random and positive asymptotic rate of growth. This is used to
establish the existence of a homogenized limit for such a diffusion when subject to Dirichlet conditions
on the boundaries of a sparse and independent array of obstacles. There is a constant effective long-time loss rate at the
obstacles. The dependence of this rate on the form and intensity of the obstacles and on the velocity field is investigated.
A Monte Carlo algorithm for the computation of the volume growth rate of the sausage is introduced and some
numerical results are presented for the Taylor--Green velocity field.
 \end{abstract}

\subparagraph{AMS subject classification (MSC2010)}60G60, 60G65, 35B27, 65C05


\section{Introduction}
We consider the problem of the existence and characterization of a homogenized limit
for 
\index{advection-diffusion|(}advection-diffusion in a 
\index{perforated domain|(}perforated domain. 
This problem was initially motivated for us as a model for the transport of
\Index{water vapour} in the atmosphere, subject to
molecular diffusion and turbulent advection, where the vapour is also lost by
condensation on suspended
\Index{ice crystals}.
It is of interest to determine the long-time rate of loss and in particular whether this is strongly affected
by the advection.
In this article we address a simple version of this set-up, where the advection is periodic in space and constant
in time and where the ice crystals remain fixed in space.

Let $K$ be a compact subset of $\R^d$ of positive
\index{Newton, I.!Newtonian capacity}Newtonian capacity. 
We assume throughout that $d\ge3$.
Let $\rho\in(0,\infty)$.
We consider eventually the limit $\rho\to0$.
Construct a random perforated domain $D\sse\R^d$ by removing all the sets $K+p$, where
$p$ runs over the support $P$ of a
\index{Poisson, S. D.!Poisson random measure}Poisson random measure $\mu$ on
$\R^d$ of intensity $\rho$.
Let $v$ be a $\Z^d$-periodic,
\index{Lipschitz, R. O. S.}Lipschitz, mean-zero, divergence-free
\Index{vector field} on $\R^d$.
Our aim is to determine the long-time behaviour, over times of order $\s^2=\rho^{-1}$, of advection-diffusion in the domain $D$
corresponding to the
\index{operator}operator\footnote{All results to follow extend to the case of the operator
$\frac12\divv a\nabla+v(x).\nabla$, where $a$ is a constant positive-definite
symmetric matrix, by a straightforward scaling transformation. We simplify the presentation by taking $a=I$. Results for the
case $a=\ve^2I$ are stated in Section \ref{DVE} for easy reference.}
$$
\cL=\tfrac12\Delta+v(x).\nabla
$$
with
\index{Dirichlet, J. P. G. L.!Dirichlet boundary condition}Dirichlet boundary
conditions. It is well known (see Section \ref{HPD})
that the long-time behaviour of advection-diffusion in the whole space $\R^d$ can be
approximated by classical, homogeneous,
\index{heat flow/loss|(}heat-flow, with a constant
\index{diffusivity}diffusivity matrix $\bar a=\bar a(v)$.
The effect of placing Dirichlet boundary conditions on the sets $K+p$
is to induce a loss of heat. 
The
\index{homogenization|(}homogenization problem in a perforated domain has been considered already
in the case of
\index{Brown, R.!Brownian motion|(}Brownian motion \cite{MR1737552},
\index{Kac, M.}\cite{MR0510113},
\index{Papanicolau, G. C.}\index{Varadhan, S. R. S.}\cite{MR609184},
\index{Rauch, J.}\index{Taylor, M.}\cite{MR0377303}
and Brownian motion with constant drift
\index{Eisele, T.}\index{Lang, R.}\cite{MR863722}.
The novelty here is to explore the possible interaction between inhomogeneity
in the drift and in the domain.
We will show that as $\rho\to0$ there exists an effective constant loss rate $\bar\l(v,K)$ in the time-scale $\s^2$.
We will also identify the limiting values of $r^{2-d}\bar\l(v,rK)$ as $r\to0$ and $r\to\infty$ and we will compute numerically this function
of $r$ for one choice of $v$ and $K$.

Fix a function $f\in L^2(\R^d)$. Write $u=u(t,x)$ for the solution to the
\index{Cauchy, A. L.!Cauchy problem}Cauchy problem for $\cL$
in $[0,\infty)\times D$ with initial data $f$, and with
Dirichlet conditions on the boundary of $D$. 
Thus, for suitably regular $K$ and $f$, $u$ is continuous on $[0,\infty)\times D$ and on $(0,\infty)\times\bar D$,
and is $C^{1,2}$ on $(0,\infty)\times D$; we have $u(0,x)=f(x)$ for all $x\in D$ and
$$
\frac{\pd u}{\pd t}=\tfrac12\D u+v(x).\nabla u\q\text{on $(0,\infty)\times D$}.
$$
We shall study the behaviour of $u$ over large scales in the limit $\rho\to0$.
Our analysis will rest on the following probabilistic representation of $u$.
Let $X$ be a
\index{diffusion!diffusion process}diffusion process in $\R^d$, independent
of $\mu$ with
\Index{generator} $\cL$ starting from $x$.
Such a process can be realised by solving the
\Index{stochastic differential equation}
\begin{equation}\label{SDE}
dX_t=dW_t+v(X_t)\,dt,\q X_0=x
\end{equation}
driven by a
Brownian motion $W$ in $\R^d$. Set
$$
T=\inf\{t\ge0:X_t\in K+P\}.
$$
Then
$$
u(t,x)=\E_x\left(f(X_t)1_{\{T>t\}}\middle|\mu\right).
$$

The key step is to express the right hand side of this identity in terms of an
analogue for $X$ of the
\index{Wiener, N.!Wiener sausage}Wiener sausage.
Associate to each path $\g\in C([0,\infty),\R^d)$ and to each interval $I\sse[0,\infty)$ a set $S_I^K(\g)\sse\R^d$ formed of the 
translates of $K$ by $\g_t$ as $t$ ranges over $I$. Thus
$$
S_I^K(\g)=\cup_{t\in I}(K+\g_t)=\{x\in \R^d: x-\g_t\in K\ {\rm for\ some}\ t\in I\}.
$$
Write $S_t^K$ for the
\Index{random set} $S_{(0,t]}^K(X)$ and write $|S_t^K|$ for the Lebesgue volume of $S_t^K$. 
We call $S_t^K$ the
\index{diffusion!diffusion sausage|(}\emph{diffusion sausage} or
\emph{$(X,K)$-sausage} and refer to $K$ as the
\undex{cross section}\emph{cross section}.
Then $T>t$ if and only if $\mu(S_t^{\hat K})=0$, where $\hat K=\{-x:x\in K\}$. 
Hence
$$
u(t,x)=\E_x\left(f(X_t)1_{\{\mu(S_t^{\hat K})=0\}}\middle|\mu\right)
$$
and so, by Fubini, we obtain the formulae
\begin{equation}\label{EU}
\E(u(t,x))=\E_x\left(f(X_t)\exp(-\rho|S_t^{\hat K}|)\right)
\end{equation}
and 
\begin{equation}\label{EUS}
\E(u(t,x)^2)=\E_x\left(f(X_t)f(Y_t)\exp(-\rho|S_t^{\hat K}(X)\cup S_t^{\hat K}(Y)|)\right)
\end{equation}
where $Y$ is an independent copy of $X$.

In the next section we review the homogenization theory for $\cL$ in the whole space.
Then, in Section \ref{KSA} we show, as a straightforward application of
\index{Kingman, J. F. C.!Kingman subadditive ergodic theorem}Kingman's
subadditive ergodic
theorem, that the sausage volume $|S_t^K|$ has almost surely an asymptotic growth rate $\g(v,K)$,
which is non-random.
In Section \ref{EDS} we make some further preparatory estimates on diffusion sausages.
Then in Section \ref{AGR} we identify the limiting values of $r^{2-d}\g(v,rK)$ as $r\to0$ and as $r\to\infty$.
In Section \ref{HH}, we use the formulae (\ref{EU}), (\ref{EUS}) to deduce the existence of a homogenized
scaling limit for the function $u$, and we prove a corresponding weak limit for the
diffusion process $X$ and the \Index{hitting time} $T$.
We shall see in particular that for large obstacles it is the effective diffusivity $\bar a$
which accounts for the loss of heat in the obstacles. On the other hand, when the obstacles are small, the loss of heat is controlled instead by the 
molecular diffusivity, even over scales where the diffusive motion itself is close to its homogenized limit.
Some results for non-unit molecular 
\index{diffusivity|(}diffusivity are recorded in Section \ref{DVE}.
Finally, in Section \ref{KZ}, we describe a new
\index{Markov, A. A.!Markov chain Monte Carlo (MCMC)}Monte Carlo algorithm to
compute the volume growth rate for the
\index{diffusion!diffusion sausage|)}$(X,K)$-sausage, 
and hence the effective long-time rate of loss of
heat\index{heat flow/loss|)}.
We present some numerical results obtained using the algorithm which interpolate between our theoretical predictions for large and small obstacles.\index{advection-diffusion|)}

\section{Review of homogenization for diffusion with periodic drift}\label{HPD}
There is a well known homogenization theory for
\index{diffusion!ldiffusion@$\cal L$-diffusion|(}$\cL$-diffusion in the whole
space $\R^d$. 
See
\index{Aronson, D. G.}\cite{MR0217444},
\index{Bensoussan, A.}\index{Lions, J.-L.}\index{Papanicolau, G. C.}\cite{MR503330},
\index{Jikov, V. V.}\index{Kozlov, S. M.}\index{Olenik, O. A.@Ole\u\i nik, O. A.}\cite{MR1329546},
\cite{MR1482931}.
We review here a few basic facts which provide the background for our treatment of the case
of a \index{perforated domain|)}perforated domain.
Our hypotheses on $v$ ensure the existence of a periodic,
\index{Lipschitz, R. O. S.|(}Lipschitz, antisymmetric
\index{tensor field|(}$2$-tensor field $\b$ on
$\R^d$ such that $\frac12\divv\b=v$. 
So we can write $\cL$ in the form
$$
\cL=\tfrac12\divv(I+\b(x))\nabla.
$$
Then $\cL$ has a continuous \index{heat kernel}heat kernel
$p:(0,\infty)\times\R^d\times\R^d$ and there exists a constant $C<\infty$,
depending only on the Lipschitz constant of $v$, such that, for all $t$, $x$ and
$y$,
\begin{equation}\label{AE}
C^{-1}\exp\{-C|x-y|^2/t\}\le p(t,x,y)\le C\exp\{-|x-y|^2/Ct\}.
\end{equation}
Moreover, $C$ may be chosen so that there also holds the following
\index{Gauss, J. C. F.!Gaussian tail estimate}Gaussian tail estimate for the
diffusion process
$X$ with \Index{generator} $\cL$ starting from $x$: for all $t>0$ and $\d>0$,
\begin{equation}\label{GTE}
\PP_x\left(\sup_{s\le t}|X_s-x|>\d\right)\le Ce^{-\d^2/Ct}.
\end{equation}

The preceding two estimates show a qualitative equivalence between $X$ and
Brownian motion, valid on all scales.
On large scales this can be refined in quantitative terms. Consider the quadratic form $q$ on $\R^d$
given by
$$
q(\xi)=\inf_{\th,\chi}\int_{\T^d}|\xi-\divv\chi+\b\nabla\th|^2\,dx
$$
where the infimum is taken over all Lipschitz functions $\th$ and all
\index{Lipschitz, R. O. S.|)}Lipschitz antisymmetric
\index{tensor field|)}$2$-tensor fields $\chi$
on the \Index{torus} $\T^d=\R^d/\Z^d$.
The infimum is achieved, so there is a positive-definite symmetric matrix $\bar a$ such that 
$$
q(\xi)=\<\xi,\bar a^{-1}\xi\>.
$$
The choice $\th=0$ and $\chi=0$ shows that $\bar a\ge I$. As the
\Index{velocity field} $v$ is scaled up, typically
it is found that $\bar a$ also becomes large. See for example
\index{Fannjiang, A.}\index{Papanicolau, G. C.|(}\cite{MR1265233} for further
discussion of this phenomenon.

We state first a deterministic homogenization result.
Let $f\in L^2(\R^d)$ and $\s\in(0,\infty)$ be given.
Denote by $u$ the solution to the
\index{Cauchy, A. L.!Cauchy problem}Cauchy problem for $\cL$ in $\R^d$ with
initial data $f(./\s)$
and set $u^{(\s)}(t,x)=u(\s^2t,\s x)$. Then
\begin{equation}\label{UUB}
\int_{\R^d}|u^{(\s)}(t,x)-\bar u(t,x)|^2\,dx\to0
\end{equation}
as $\s\to\infty$, for all $t\ge0$, where $\bar u$ is the solution to the Cauchy problem for $\tfrac12\divv\bar a\nabla$ in $\R^d$
with initial data $f$.

In probabilistic terms, we may fix $x\in\R^d$ and $\s\in(0,\infty)$ and
consider the
\index{diffusion!ldiffusion@$\cal L$-diffusion|)}$\cL$-diffusion process
$X$ starting from $\s x$. Set $X^{(\s)}_t=\s^{-1}X_{\s^2t}$. Then it is known 
\index{Papanicolau, G. C.|)}\index{Varadhan, S. R. S.}\cite{MR0712714} that
\begin{equation}\label{XXB}
X^{(\s)}\to\bar X,\q\text{weakly on $C([0,\infty),\R^d)$}
\end{equation}\undex{weak convergence}
where $\bar X$ is a Brownian motion in $\R^d$ with
\index{diffusivity|)}diffusivity $\bar a$ starting from $x$.
The two
\index{homogenization|)}homogenization statements are essentially equivalent given the regularity implicit
in the above qualitative estimates, the
\index{Markov, A. A.!Markov property|(}Markov property, and the identity
$$
u^{(\s)}(t,x)=\E(f(X_t^{(\s)})).
$$

\section{Existence of a volume growth rate for a diffusion sausage with
periodic drift}\label{KSA}\index{diffusion!diffusion sausage|(}
Recall that the drift $v$ is $\Z^d$-periodic and divergence-free.
\begin{theorem}\label{VGR}
There exists a constant $\g=\gamma(v,K)\in(0,\infty)$ such that, for all $x$,
\[
\lim_{t\to\infty}\frac{|S_t^K|}t=\gamma,\q\text{$\PP_x$-almost surely.}
\]
\end{theorem}
\begin{proof}
Write $\pi$ for the \Index{projection} $\R^d\to\T^d$. Since $v$ is periodic, the projected process $\pi(X)$ is a diffusion on $\T^d$. 
As $v$ is divergence-free, the unique invariant distribution for $\pi(X)$ on $\T^d$ is the uniform distribution. 
The lower bound in (\ref{AE}) shows that the transition density of $\pi(X_1)$ on $\T$ is uniformly positive. 
By a standard argument $\pi(X)$ is therefore uniformly and
\index{ergodicity!geometric ergodicity}geometrically ergodic. 
Consider the case where $X_0$ is chosen randomly, and independently of $W$, such that $\pi(X_0)$ is uniformly distributed 
on $\T^d$. Then $\pi(X)$ is stationary.
For integers $0\le m<n$, define $V_{m,n}=|S^K_{(m,n]}|$. Then $V_{l,n}\le V_{l,m}+V_{m,n}$ whenever $0\le l<m<n$. Since Lebesgue measure
is translation invariant and $\pi(X)$ is stationary, the distribution of the array $(V_{m+k,n+k}:0\le m<n)$ is the same for all $k\ge0$.
Moreover $V_{m,n}$ is integrable for all $m$, $n$ by standard diffusion estimates. Hence by the
\index{Kingman, J. F. C.!Kingman subadditive ergodic theorem}subadditive
ergodic theorem \cite{MR0254907} we can conclude that, for some constant $\g\ge0$,
\[
\lim_{n\to \infty}\frac{|S^K_n|}n=\gamma,\q\text{almost surely.}
\]
The positivity of $\g$ follows from the positivity of $\capp(K)$ using Theorem \ref{AGRT} below.

Let $\PP_x$ be the probability measure on $C([0,\infty),\R^d)$ which is the law of the process $X$ starting from $x$. 
Set 
$$
g(x)=\PP_x\left(\lim_{n\to \infty}\frac{|S^K_n|}n=\gamma\right),\q
\tilde g(x)=\PP_x\left(\lim_{n\to \infty}\frac{|S^K_{(1,n]}|}n=\gamma\right).
$$
Then $g$ is periodic and $\tilde g=g$. We have shown that 
\[
\int_{x\in[0,1]^d}g(x)\,dx=1.
\]
Hence $g(x)=1$ for Lebesgue-almost-all $x$. But then by the
Markov property, for every $x$,
\[
g(x)=\tilde g(x)=\int_{\R^d}p(1,x,y)g(y)\,dy=1
\]
which is the desired
\undex{almost sure convergence@almost-sure convergence}almost-sure convergence
for discrete parameter $n$. An obvious monotonicity
argument extends this to the continuous parameter $t$.
\end{proof}

\section{Estimates for the diffusion sausage}\label{EDS}
We prepare some estimates on the diffusion sausage which will be needed later. These are
of a type well known for 
Brownian motion
\index{Le Gall, J.-F.}\cite{MR866344} and extend in a straightforward way using
the qualitative Gaussian bounds (\ref{AE}) and (\ref{GTE}).

\begin{lemma}\label{LPB}
For all $p\in[1,\infty)$ there is a constant $C(p,v,K)<\infty$ such that, for all $t\ge0$
and all $x\in\R^d$,
$$
\E_x\left(|S_t^K|^p\right)^{1/p}\le C(t+1).
$$
\end{lemma}
\begin{proof} 
Reduce to the case $t=1$ by \Index{subadditivity} of volume and
\index{lpnorm@$L^p$ norm}$L^p$-norms and by the
Markov property. The estimate
then follows from (\ref{GTE}) since
$$
|S^K_1|\le \o_d\left(\rad(K)+\sup_{t\le 1}|X_t-x|\right)^d.
$$
\end{proof}

\begin{lemma}\label{INT}
There is a constant $C(v,K)<\infty$ with the following property.
Let $X$ and $Y$ be independent
\index{diffusion!ldiffusion@$\cal L$-diffusion}$\cL$-diffusions starting
from $x$.
For all $t\ge1$ and all $x\in\R^d$, for all $a,b\ge0$,
\begin{equation}\label{XYI}
\PP_x\left(S_{(at,(a+1)t]}^K(X)\cap S_{(bt,(b+1)t]}^K(Y)\not=\es\right)\le C(a+b)^{-d/2}
\end{equation}
and, when $b\ge a+1$,
\begin{equation}\label{XXI}
\PP_x\left(S_{(at,(a+1)t]}^K(X)\cap S_{(bt,(b+1)t]}^K(X)\not=\es\right)\le C(b-a-1)^{-d/2}.
\end{equation}
\end{lemma}
\begin{proof}
We write the proof for the case $t=1$. The same argument applies generally. There is alternatively
a reduction to the case $t=1$ by scaling. Assume that $b\ge a+1$. 
Write $\cF_t$ for the $\s$-algebra generated by $(X_s:0\le s\le t)$ and set
$$
R_a=\sup_{a\le t\le a+1}|X_t-X_{a+1}|,\q
R_b=\sup_{b\le t\le b+1}|X_t-X_b|,\q
Z=X_b-X_{a+1}.
$$
Then, by (\ref{GTE}),
$$
\PP_x(R_b\ge|Z|/3|\cF_b)\le Ce^{-|Z|^2/9C}
$$
so, using (\ref{AE}),
$$
\PP_x(R_b\ge|Z|/3)\le C\E_x(e^{-|Z|^2/9C})\le C(b-a-1)^{d/2}.
$$
On the other hand, by (\ref{AE}) again,
$$
\PP_x(R_a\ge|Z|/3|\cF_{a+1})\le C(b-a-1)^{d/2}R_a^d
$$
so, using (\ref{GTE}),
$$
\PP_x(R_a\ge|Z|/3)\le C(b-a-1)^{d/2}.
$$
Moreover (\ref{AE}) gives also
$$
\PP_x(2\rad(K)\ge|Z|/3)\le C(b-a-1)^{d/2}.
$$
Now if $S_{(a,a+1]}^K(X)\cap S_{(b,b+1]}^K(X)\not=\es$ then either $R_a\ge|Z|/3$ or $R_b\ge|Z|/3$ or $2\rad(K)\ge|Z|/3$.
Hence the preceding estimates imply (\ref{XXI}).
The proof of (\ref{XYI}) is similar, resting on the fact that $X_a-Y_b$ has density bounded by $C(a+b)^{-d/2}$,
and is left to the reader.
\end{proof}

\begin{lemma}\label{ESL}
As $t\to\infty$, we have
$$
\sup_x\E_x\left(\left|\frac{|S_t^K|}t-\g\right|\right)\to0
$$
and 
$$
\sup_x\E_x\left(\frac{|S^K_t(X)\cap S_t^K(Y)|}t\right)\to0
$$
where $Y$ is an independent copy of $X$.
\end{lemma}
\begin{proof}
Note that 
$$
|S_{(1,t+1]}^K|\le|S_{t+1}^K|\le |S_1^K|+|S_{(1,t+1]}^K|.
$$
Given Lemma \ref{LPB}, the first assertion will follow if we can show that, as $t\to\infty$,
$$
\sup_x\E_x\left(\left|\frac{|S_{(1,t+1]}^K|}t-\g\right|\right)\to0.
$$
But by the
\index{Markov, A. A.!Markov property|)}Markov property and using (\ref{AE}),
\begin{align*}
&\E_x\left(\left|\frac{|S_{(1,t+1]}^K|}t-\g\right|\right)
=\int_{\R^d}p(1,x,y)\E_y\left(\left|\frac{|S_t^K|}t-\g\right|\right)\,dy\\
&\q\q\le C\int_{[0,1]^d}\E_y\left(\left|\frac{|S_t^K|}t-\g\right|\right)\,dy\to0
\end{align*}
as $t\to\infty$, where we used the
\undex{almost sure convergence@almost-sure convergence}almost-sure
convergence $|S_t^K|/t\to\g$ when $\pi(X_0)$ is uniform, together with
\Index{uniform integrability} from Lemma \ref{LPB}
to get the final limit.

For the second assertion,
choose $q\in(1,3/2)$ and $p\in(3,\infty)$ with $1/p+1/q=1$. Then, for $j,k\ge0$, by Lemmas \ref{LPB}
and \ref{INT}, there is a constant $C(p,v,K)<\infty$ such that
\begin{align*}
&\E_x(|S_{(j,j+1]}^K(X)\cap S_{(k,k+1]}^K(Y)|)\\
&\q\q\le\E_x\left(|S_{(j,j+1]}^K(X)|1_{\{S_{(j,j+1]}^K(X)\cap S_{(k,k+1]}^K(Y)\not=\es\}}\right)\\
&\q\q\le\E_x\left(|S_1^K(X)|^p\right)^{1/p}\PP_x\left(S_{(j,j+1]}^K(X)\cap S_{(k,k+1]}^K(Y)\not=\es\right)^{1/q}\\
&\q\q\le C(j+k)^{-d/2q}.
\end{align*}
So, as $n\to\infty$, we have
\begin{align*}
&\E_x\left(\frac{|S^K_n(X)\cap S_n^K(Y)|}n\right)\\
&\q\q\le\frac{\E_x(|S_1^K(X)|)}n
           +\sum_{j=1}^{n-1}\sum_{k=0}^{n-1}\frac{\E_x(|S_{(j,j+1]}^K(X)\cap S_{(k,k+1]}^K(Y)|)}n\\
&\q\q\le Cn^{-1}+Cn^{-d/(2q)+1}\to0.
\end{align*}
\end{proof}

\section{Asymptotics of the growth rate for small and large cross-sections}\label{AGR}
We investigate the behaviour of the asymptotic growth rate $\g(v,rK)$ of the volume 
of the $(X,rK)$-sausage $S^{rK}_t$ in the limits $r\to0$ and $r\to\infty$.
Recall the
\Index{stochastic differential equation} (\ref{SDE}) for $X$ and recall the
rescaled process $X^{(\s)}$
from Section \ref{HPD}.
Set $W^{(\s)}_t=\s^{-1}W_{\s^2t}$. Then $W^{(\s)}$ is also a 
Brownian motion and $X^{(\s)}$
satisfies the stochastic differential equation
\beq
dX^{(\s)}_t=dW^{(\s)}_t+v^{(\s)}(X^{(\s)}_t)\,dt
\label{sde1}
\eeq
where $v^{(\s)}(x)=\s v(\s x)$.
This makes it clear that $X^{(\s)}\to W$ as $\s\to0$ weakly on $C([0,\infty),\R^d)$. Recall from Section \ref{HPD} the fact that 
$X^{(\s)}\to\bar X$ as $\s\to\infty$, in the same sense, where $\bar X$ is a Brownian motion with \index{diffusivity|(}diffusivity $\bar a$.

Take $\s=r$. Then 
$$
S^{rK}_t(X)=rS_{r^{-2}t}^K(X^{(r)})
$$
so
$$
|S^{rK}_t(X)|=r^d |S_{r^{-2}t}^K(X^{(r)})|.
$$
Hence the limit 
$$
\g(v^{(r)},K):=\lim_{t\to\infty}|S_t^K(X^{(r)})|/t
$$
exists and equals $r^{2-d}\g(v,rK)$. 
The weak limits for $X^{(r)}$ as $r\to0$ or $r\to\infty$ suggest the following result,
which however requires further argument
because the asymptotic growth rate of the sausage is not a continuous function on $C([0,\infty),\R^d)$.
Write $\capp(K)$ for the
\index{Newton, I.!Newtonian capacity}Newtonian capacity of $K$ and
$\capp_{\bar a}(K)$ for the \Index{capacity} of $K$ with respect to the
\index{diffusivity|)}diffusivity matrix $\bar a$. Thus
$$
\capp_{\bar a}(K)=\sqrt{\det\bar a}\,\capp (\bar a^{-1/2}K). 
$$

\begin{theorem}\label{AGRT}
We have
$$
\lim_{r\to 0}r^{2-d}\g(v,rK)=\lim_{r\to0}\g(v^{(r)},K)=\g(0,K)=\capp(K)
$$
and
$$
\lim_{r\to\infty}r^{2-d}\g(v,rK)=\lim_{r\to\infty}\g(v^{(r)},K)=\capp_{\bar a}(K).
$$
\end{theorem}
\begin{proof}
Fix $T\in(0,\infty)$ and write $I(j)$ for the interval $((j-1)T,jT]$. Consider for $1\le j\le k$ the function $F_{j,k}$ on $C([0,\infty),\R^d)$ defined by
$$
F_{j,k}(\g)=|S_{I(j)}^{K}(\g)\cap S_{I(k)}^{K}(\g)|/T.
$$
Then $F_{j,k}$ is continuous, so 
$$
\lim_{r\to0}\E(F_{j,k}(X^{(r)}))=\E(F_{j,k}(W)),\q
\lim_{r\to\infty}\E(F_{j,k}(X^{(r)}))=\E(F_{j,k}(\bar X)).
$$
Choose $X_0$ so that $\pi(X_0)$ is uniformly distributed. Then by stationarity $\E(F_{j,k}(X^{(r)}))=\E(F_{k-j}(X^{(r)}))$
where $F_j=F_{1,j+1}$. Fix $r$ and write $S_I^K(X^{(r)})=S_I$. Note that 
$$
S_{(0,nT]}=S_{I(1)}\cup\dots\cup S_{I(n)}.
$$
So, by \Index{inclusion-exclusion}, we obtain
$$
n\E(F_0(X^{(r)}))-\sum_{j=1}^{n-1}(n-j)\E(F_j(X^{(r)}))\le\E(|S_{(0,nT]}|/T)\le n\E(F_0(X^{(r)})).
$$
Divide by $n$ and let $n\to\infty$ to obtain
$$
\E(F_0(X^{(r)}))-\sum_{j=1}^{\infty}\E(F_j(X^{(r)}))\le \g(v^{(r)},K)\le\E(F_0(X^{(r)})).
$$
Fix $q\in(1,3/2)$ and $p\in(3,\infty)$ with $p^{-1}+q^{-1}=1$.
By Lemmas \ref{LPB} and \ref{INT}, there is a constant $C(p,v,K)<\infty$ such that, for all $r$ and $j$,
\begin{align*}
&\E(F_j(X^{(r)}))
=\E\left(|S_{I(1)}\cap S_{I(j+1)}|/T\right)\\
&\q\q\le\E\left(\left||S_{I(1)}|/T\right|^p\right)^{1/p}
\PP\left(S_{I(1)}\cap S_{I(j+1)}\not=\es\right)^{1/q}
\le2C(j-1)^{-d/2q}.
\end{align*}
Given $\ve>0$, we can choose $J(p,v,K)<\infty$ so that 
\begin{equation}\label{SJK}
\sum_{j=J+1}^{\infty}\E(F_j(X^{(r)}))\le2C\sum_{j=J}^\infty j^{-d/2q}\le\ve.
\end{equation}
We follow from this point the case $r\to\infty$. The argument for the other limit is the same. Let $r\to\infty$ to obtain
\begin{align}\notag
\E(F_0(\bar X))-\sum_{j=1}^J\E(F_j(\bar X))-\ve&\le\liminf_{r\to\infty}\g(v^{(r)},K)\\ \label{EFO}
&\le\limsup_{r\to\infty}\g(v^{(r)},K)\le\E(F_0(\bar X)).
\end{align}
It is known that 
$$
\lim_{T\to\infty}\E(F_0(\bar X))=\lim_{T\to\infty}\E(|S_T^K(\bar X)|/T)=\capp_{\bar a}(K).
$$
See \cite{MR866344} for the case $\bar a=I$. The 
general case follows by a scaling transformation. Note that, for $j\ge1$,
$$
|S_{(0,T]}^K(\bar X)\cap S_{(jT,(j+1)T]}^K(\bar X)|+|S_{(0,(j+1)T]}^K(\bar X)|\le\sum_{i=1}^{j+1}|S_{((i-1)T,iT]}^K(\bar X)|.
$$
Take expectation, divide by $T$ and let $T\to\infty$ to obtain
\begin{align*}
&(j+1)\capp_{\bar a}(K)+\limsup_{T\to\infty}\E|S_{(0,T]}^K(\bar X)\cap
    S_{(jT,(j+1)T]}^K(\bar X)|/T\\
&\qquad{}\le (j+1)\capp_{\bar a}(K)
\end{align*}
which says exactly that
$$
\lim_{T\to\infty}\E(F_j(\bar X))=0.
$$
Hence the desired limit follows on letting $T\to\infty$ in (\ref{EFO}).
\end{proof}

\section{Homogenization of the advection-diffusion equation in a perforated domain}\label{HH}\index{advection-diffusion|(}\index{homogenization|(}
Our main results are analogues to the homogenization statements (\ref{UUB}), (\ref{XXB}) for advection-diffusion 
in a \index{perforated domain}perforated domain.
Recall that $v$ is a $\Z^d$-periodic,
\index{Lipschitz, R. O. S.}Lipschitz, mean-zero, divergence-free
\index{vector field}vector field on $\R^d$,
and $K$ is a compact subset of $\R^d$. 
The domain $D\sse\R^d$ is constructed by removing all the sets $K+p$, where
$p$ runs over the set $P$ of atoms of a
\index{Poisson, S. D.!Poisson random measure}Poisson random measure $\mu$ on
$\R^d$ of intensity $\rho=\s^{-2}$.
Write
$$
\bar a=\bar a(v,K),\q \bar\l=\bar\l(v,K)=\g(v,\hat K).
$$

\begin{theorem}
Let $f\in L^2(\R^d)$ and $\s\in(0,\infty)$ be given.
Denote by $u$ the solution\footnote{We extend $u$ to a function on
$[0,\infty)\times\R^d$ by setting $u(t,x)=0$ for any $x\not\in D$.} to the
\index{Cauchy, A. L.!Cauchy problem|(}Cauchy problem for 
$$
\cL=\tfrac12\Delta+v(x).\nabla
$$
in $[0,\infty)\times D$ with initial data $f(\cdot/\s)$, and with
\index{Dirichlet, J. P. G. L.!Dirichlet boundary condition}Dirichlet
conditions on the boundary of $D$.
Set $u^{(\s)}(t,x)=u(\s^2t,\s x)$. Then
\begin{equation*}
\E\int_{\R^d}|u^{(\s)}(t,x)-\bar u(t,x)|^2\,dx\to0
\end{equation*}
as $\s\to\infty$, for all $t\ge0$, where $\bar u$ is the solution to the Cauchy problem for $\frac12\divv\bar a\nabla-\bar\l$
in $[0,\infty)\times\R^d$ with initial data $f$.
\end{theorem}
\begin{proof}
Replace $t$ by $\s^2t$, $x$ by $\s x$ and $f$ by $f(\cdot/\s)$ in (\ref{EU}) and (\ref{EUS}) to obtain
\begin{equation*}
\E\left(u^{(\s)}(t,x)\right)=\E_{\s x}\left(f(X_t^{(\s)})\exp\{-\rho|S_{\s^2 t}^{\hat K}(X)|\}\right)
\end{equation*}
and
\begin{equation*}
\E\left(u^{(\s)}(t,x)^2\right)=\E_{\s x}\left(f(X_t^{(\s)})f(Y_t^{(\s)})\exp\{-\rho|S_{\s^2 t}^{\hat K}(X)\cup S_{\s^2 t}^{\hat K}(Y)|\}\right)
\end{equation*}
where the subscript $\s x$ specifies the starting point of $X$ and where $Y$ is an independent copy of $X$.
We omit from now on the superscript $\hat K$.
Then\footnote{This is an instance of the formula $\E(|X-a|^2)=\var(X)+(\E(X)-a)^2$.}
\begin{align}\notag
&\E\left(|u^{(\s)}(t,x)-\bar u(t,x)|^2\right)\\\notag
&=\E_{\s x}\left(f(X_t^{(\s)})f(Y_t^{(\s)})e^{-\rho|S_{\s^2 t}(X)\cup S_{\s^2 t}(Y)|}(1-e^{-\rho|S_{\s^2 t}(X)\cap S_{\s^2 t}(Y)|})\right)\\\notag
&\q\q+\left(\E_{\s x}(f(X_t^{(\s)})\{e^{-\rho|S_t(X)|}-e^{-\bar\l t}\}\right.\\\notag
&\q\q\q\q            +\left.\left\{\E_{\s x}(f(X_t^{(\s)}))-\E_x(f(\bar X_t))\right\}e^{-\bar\l t}\right)^2\\\notag
&\le\E_{\s x}\left(f(X_t^{(\s)})^2\right)\E_{\s x}\left(\rho|S_{\s^2 t}(X)\cap S_{\s^2 t}(Y)|\right)^{1/2}\\\notag
&\q\q+2\E_{\s x}\left(f(X_t^{(\s)})^2\right)\E_{\s x}\left(\rho|S_{\s^2 t}(X)|-\bar\l t|\right)\\\label{EUU}
&\q\q\q\q+2|u_0^{(\s)}(t,x)-\bar u_0(t,x)|^2
\end{align}
where $u_0^{(\s)}$ and $\bar u_0$ denote the corresponding solutions to the
\index{Cauchy, A. L.!Cauchy problem|)}Cauchy problem with initial data $f$ in
the whole space,
and we used Cauchy--Schwarz and $(a+b)^2\le2a^2+2b^2$ and $|e^{-a}-e^{-b}|^2\le|b-a|$ to obtain the inequality.
Now
$$
\int_{\R^d}\E_{\s x}\left(f(X_t^{(\s)})^2\right)\,dx=\int_{\R^d}|f(x)|^2\,dx<\infty
$$
because $dx$ is stationary for $X^{(\s)}$ and, by (\ref{UUB}), as $\s\to\infty$
$$
\int_{\R^d}|u_0^{(\s)}(t,x)-\bar u_0(t,x)|^2\,dx\to0.
$$
So, using Lemma \ref{ESL}, on integrating (\ref{EUU}) over $\R^d$ and letting
$\s\to\infty$, we conclude that the right-hand side tends to $0$,
proving the theorem.
\end{proof}

\begin{theorem}
Let $x\in\R^d$ and $\s\in(0,\infty)$ be given.
Let $X$ be an
\index{diffusion!ldiffusion@$\cal L$-diffusion|(}$\cL$-diffusion in $\R^d$
starting from $\s x$ and set
$$
T=\inf\{t\ge0:X_t\in K+P\}.
$$
Set $X^{(\s)}_t=\s^{-1}X_{\s^2t}$ and $T^{(\s)}=\s^{-2}T$.
Write $\bar X$ for a 
Brownian motion in $\R^d$ with
\index{diffusivity|(}diffusivity $\bar a$ starting from $x$,
and write $\bar T$ for an exponential random variable of parameter $\bar\l$, independent of $\bar X$.
Then, as $\s\to\infty$,
$$
(X^{(\s)},T^{(\s)})\to(\bar X,\bar T),\q\text{weakly on $C([0,\infty),\R^d)\times[0,\infty)$.}
$$
\end{theorem}
\begin{proof}
Write $S_t$ for the $(X,\hat K)$-sausage.
Fix a bounded continuous function $F$ on $C([0,\infty),\R^d)$ and fix $t>0$. Then
$$
\E\left(F(X^{(\s)})1_{\{T^{(\s)}>t\}}\right)=\E\left(F(X^{(\s)})\exp\{-\rho|S_{\s^2t}|\}\right)
$$
and 
$$
\E\left(F(\bar X)1_{\{\bar T>t\}}\right)=\E\left(F(\bar X)e^{-\bar\l t}\right)
$$
so
\begin{align*}
&\left|\E\left(F(X^{(\s)})1_{\{T^{(\s)}>t\}}\right)-\E\left(F(\bar X)1_{\{\bar T>t\}}\right)\right|\\
&\q\le\|F\|_\infty\,\E_{\s x}|\rho|S_{\s^2t}|-\bar\l t|+|\E(F(X^{(\s)})-\E(F(\bar X))|e^{-\bar\l t}.
\end{align*}
On letting $\s\to\infty$, the first term on the right tends to $0$ by Lemma \ref{ESL} and the second term tends to $0$ by (\ref{XXB}),
so the left hand side also tends to $0$, proving the theorem.
\end{proof}\index{advection-diffusion|)}\index{homogenization|)}\index{diffusion!diffusion sausage|)}

\section{The case of diffusivity $\ve^2I$}\label{DVE}
In this section and the next we fix $\ve\in(0,\infty)$ and consider the more general case of the \Index{operator}
$$
\cL=\tfrac12\ve^2\D+v(x).\nabla.
$$
The following statements follow from the corresponding statements ab\-ove for the case $\ve=1$ by scaling.
Fix $x\in\R^d$ and let $X$ be an
\index{diffusion!ldiffusion@$\cal L$-diffusion|)}$\cL$-diffusion in $\R^d$ starting from $x$. Then
$$
|S_t^K(X)|/t\to\g(\ve,v,K),\q\PP_x\text{-almost surely}
$$
as $t\to\infty$, where
$$
\g(\ve,v,K)=\ve^2\g(\ve^{-2}v,K).
$$
Moreover, setting $v^{(r)}(x)=rv(rx)$, as above, we have
$$
\g(\ve,v^{(r)},K)\to\capp_{\ve^2I}(K)=\ve^2\capp(K),\q\text{as $r\to0$}
$$
and
$$
\g(\ve,v^{(r)},K)\to\capp_{\bar a(\ve,v)}(K),\q\text{as $r\to\infty$}
$$
where
$$
\bar a(\ve,v)=\ve^2\bar a(\ve^{-2}v).
$$
Fix $\s\in(0,\infty)$ and suppose now that $X$ starts at $\s x$. 
Define as above $T=\inf\{t\ge0:X_t\in K+P\}$ and write
$X^{(\s)}_t=\s^{-1}X_{\s^2t}$ and $T^{(\s)}=\s^{-2}T$.
Then, as $\s\to\infty$,
$$
(X^{(\s)},T^{(\s)})\to(\bar X,\bar T),\q\text{weakly on $C([0,\infty),\R^d)\times[0,\infty)$} 
$$
where $\bar X$ is a 
\index{Brown, R.!Brownian motion|)}Brownian motion in $\R^d$ of
\index{diffusivity|)}diffusivity $\bar a(\ve,v)$
starting from $x$, and where $\bar T$ is an exponential random variable 
independent of $\bar X$, of parameter $\bar\l(\ve,v,K)=\g(\ve,v,\hat K)$. 

\section{Monte Carlo computation of the asymptotic growth rate}\label{KZ}
\index{Markov, A. A.!Markov chain Monte Carlo (MCMC)|(}Let $X$ be as in the
preceding section.
Fix $T\in(0,\infty)$.
The following algorithm may be used to estimate numerically the volume of the $(X,K)$-sausage $S_T=S_T^K(X)$.
The algorithm is determined by the choice of three parameters $N$, $m$,
$J\in\N$.\index{Euler, L.!Euler Maruyama solution@Euler--Maruyama solution}
\begin{itemize}
\item \bf{Step 1:} \normalfont \emph{Compute an
Euler--Maruyama solution
  $(X_{n\D t}:n=0,1,\dots,\allowbreak N)$ to the
\Index{stochastic differential equation}
\begin{equation}\label{SDEE}
dX_t=\ve dW_t+v(X_t)\,dt,\q X_0=x
\end{equation}
up to the final time $T=N\D t$} (Figure \ref{easy_money}a).
\item \bf{Step 2:} \normalfont \emph{Calculate }
\[
R_K=\max_{y\in K,\, 1\le k\le d}|y^k|,\q R_{X,T}=\max_{1\le n\le N,\,1\le k\le d} |X^k_{n\D t}-x^k|.
\]
We approximate $S_T$  by $S_T^{(N)}=\cup_{0\le n\le N}(K+X_{n\D t})$.
Note that $S_T^{(N)}$ is contained in the cube with side-length $L=2(R_K+R_{X,T})$ centred at $x$ 
(Figure \ref{easy_money}b). 
\begin{figure}[htb]  
\begin{center}
\subfigure[Step 1]{\epsfig{file=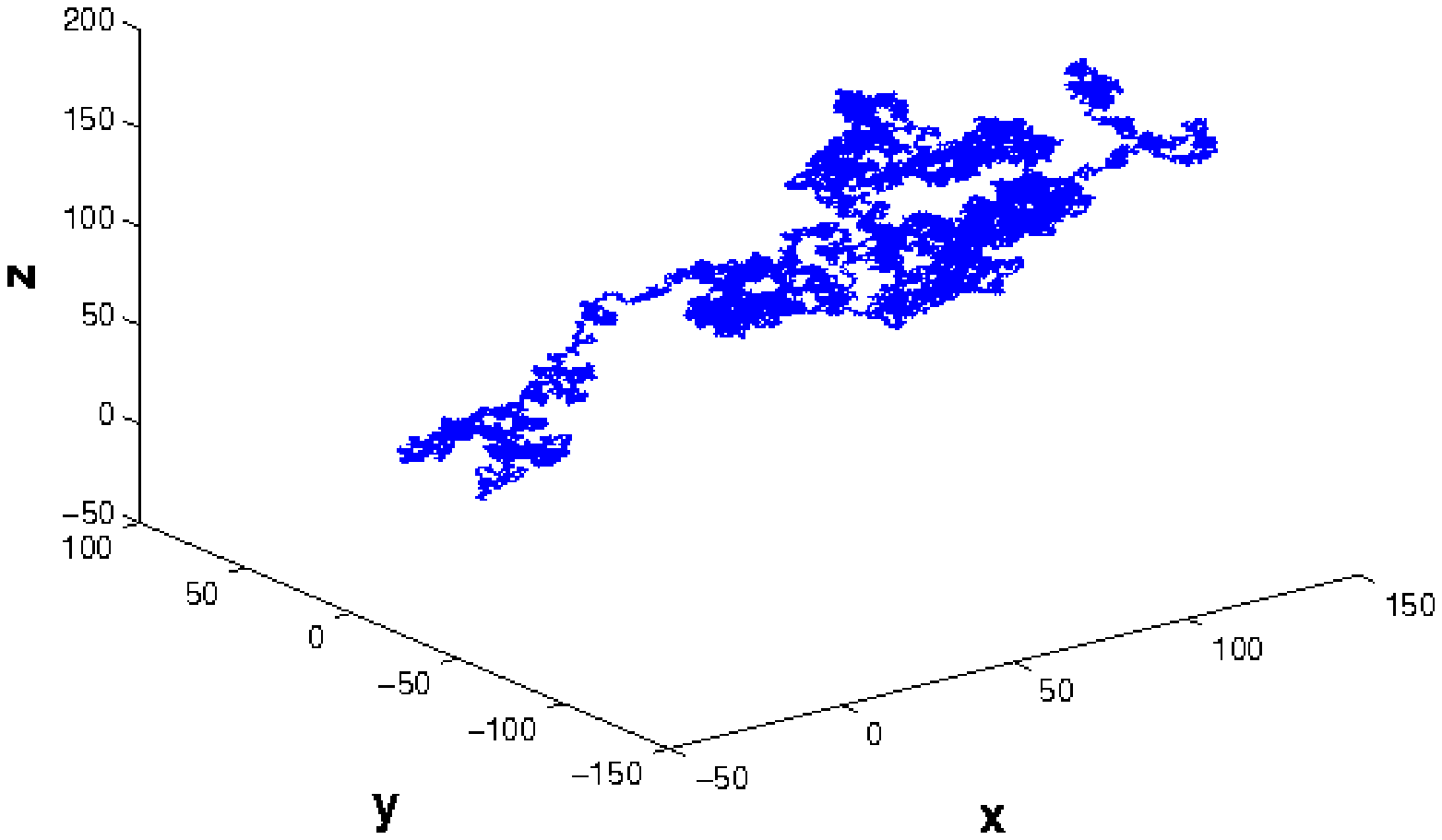,scale=0.3}}
\subfigure[Step 2]{\epsfig{file=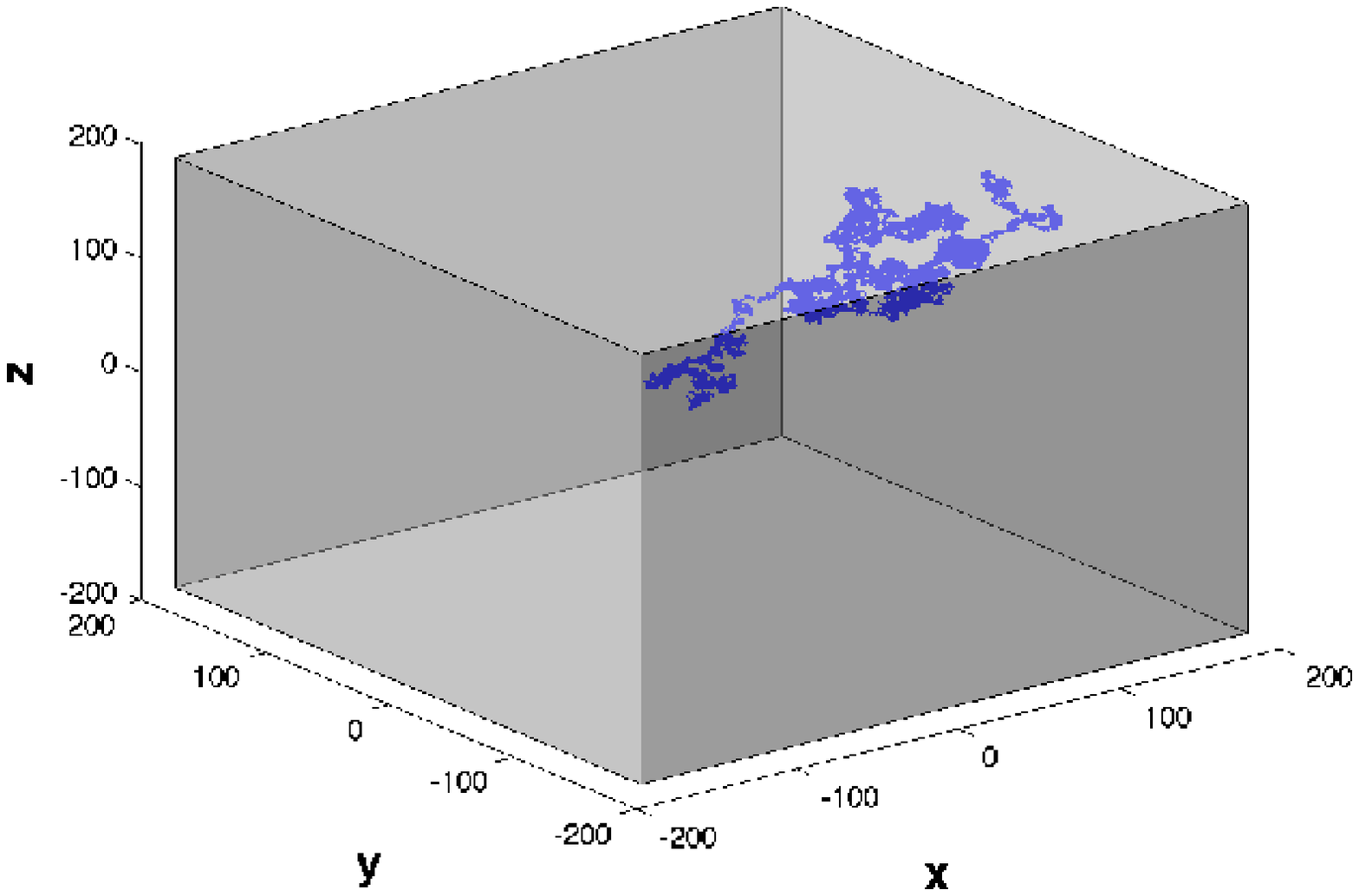,scale=0.3}} 
\caption{First two steps of the algorithm}
\label{easy_money}
\end{center}
\end{figure}
\item \bf{Step 3:} \normalfont \emph{Subdivide the cube of side-length $L$ centred at $x$ into $2^d$ sub-cubes of side-length $L/2$ 
and check which of them have non-empty intersection with $S_T^{(N)}$. 
Discard any sub-cubes with empty intersection. 
Repeat the division and discarding procedure in each of the remaining sub-cubes (Figure \ref{easy_money1})
iteratively to obtain $I$ sub-cubes of side-length $L/2^m$, centred at $y_1$,
\dots, $y_I$ say, whose union
contains $S_T^{(N)}$.}
\begin{figure}[htb]  
\begin{center}
\subfigure[$m=1$, $I=8$]{\epsfig{file=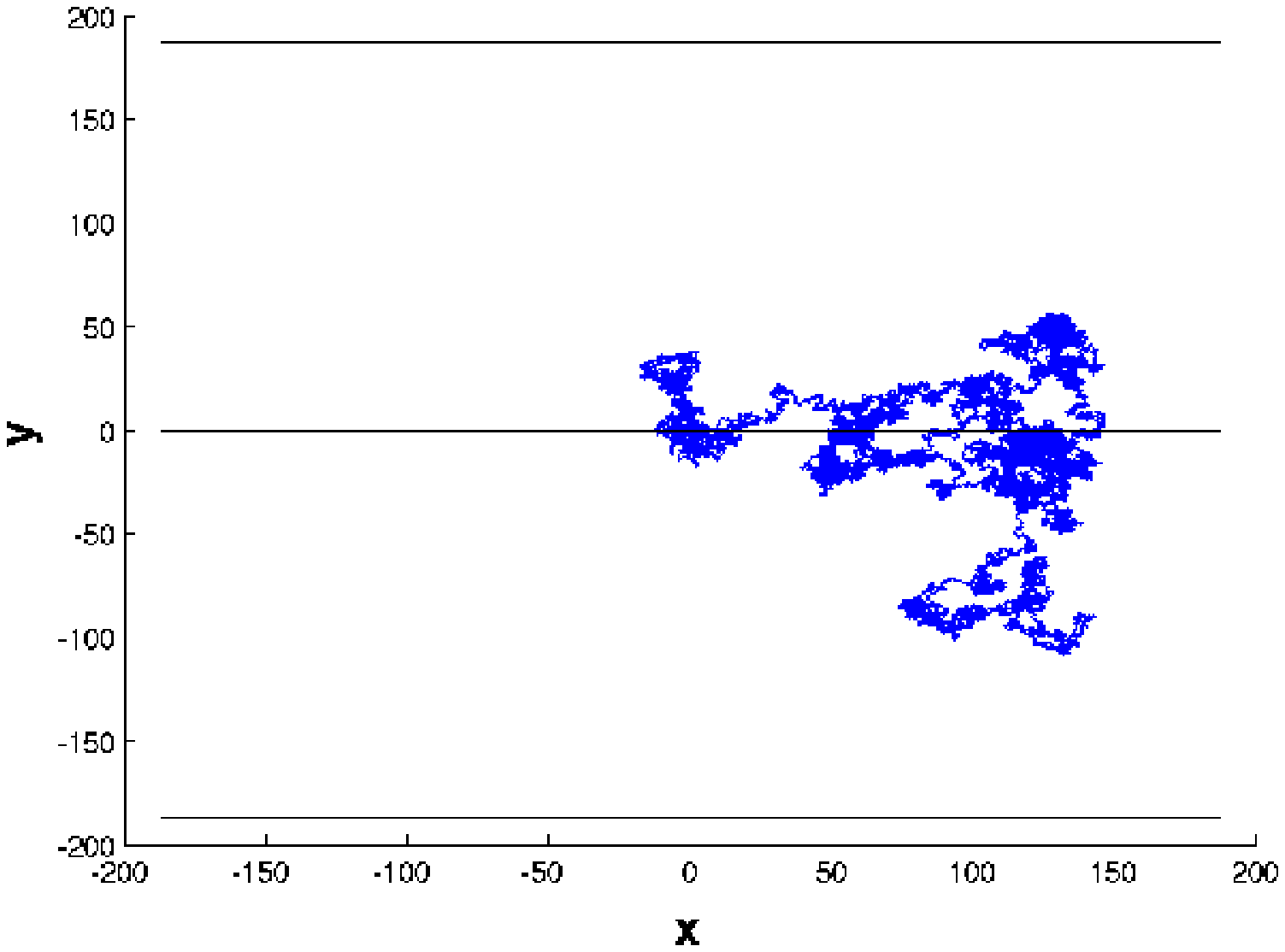,scale=0.3}}
\subfigure[$m=2$, $I=15$]{\epsfig{file=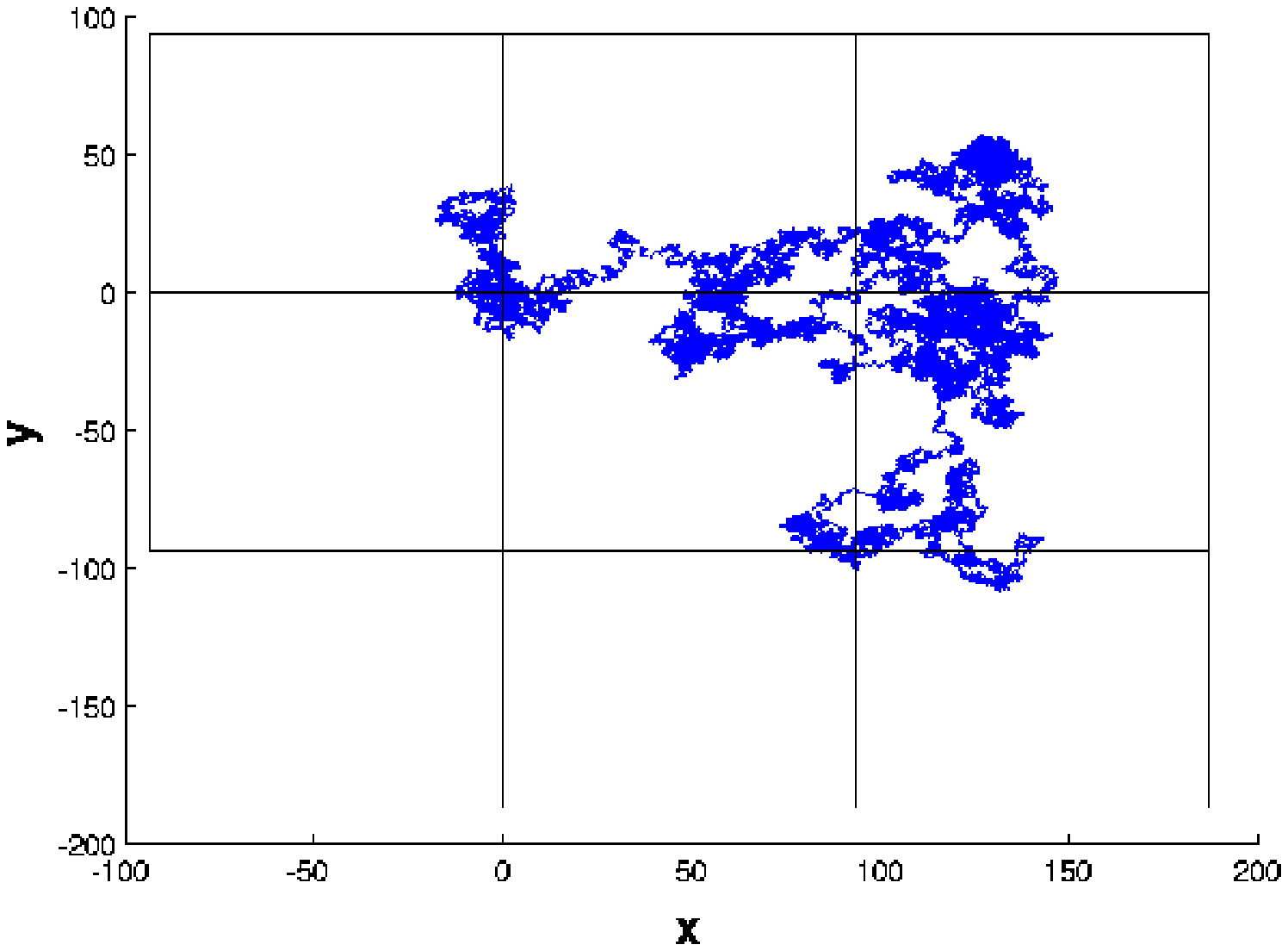,scale=0.3}} 
\subfigure[$m=3$, $I=40$]{\epsfig{file=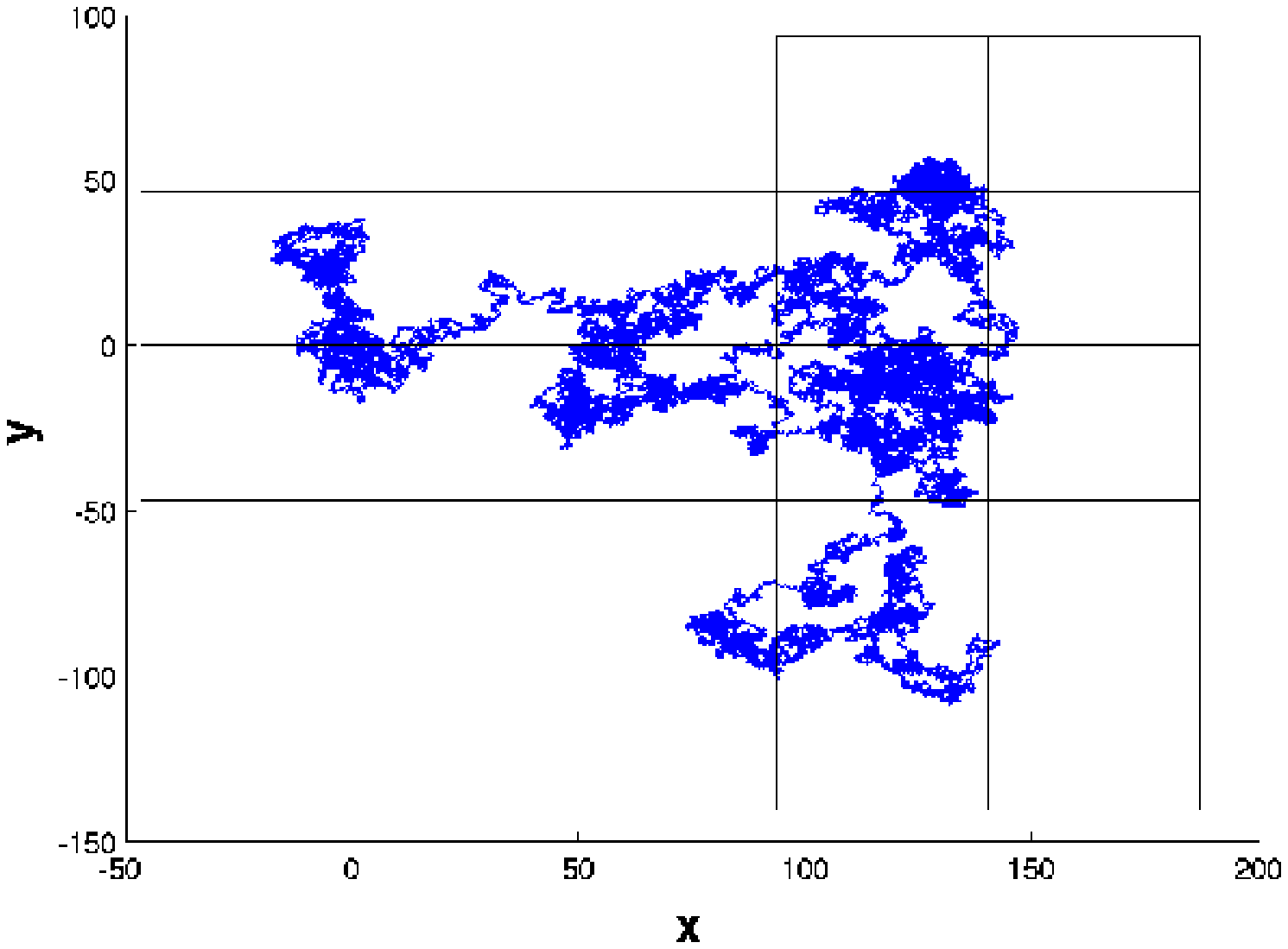,scale=0.3}}
\subfigure[$m=4$, $I=111$]{\epsfig{file=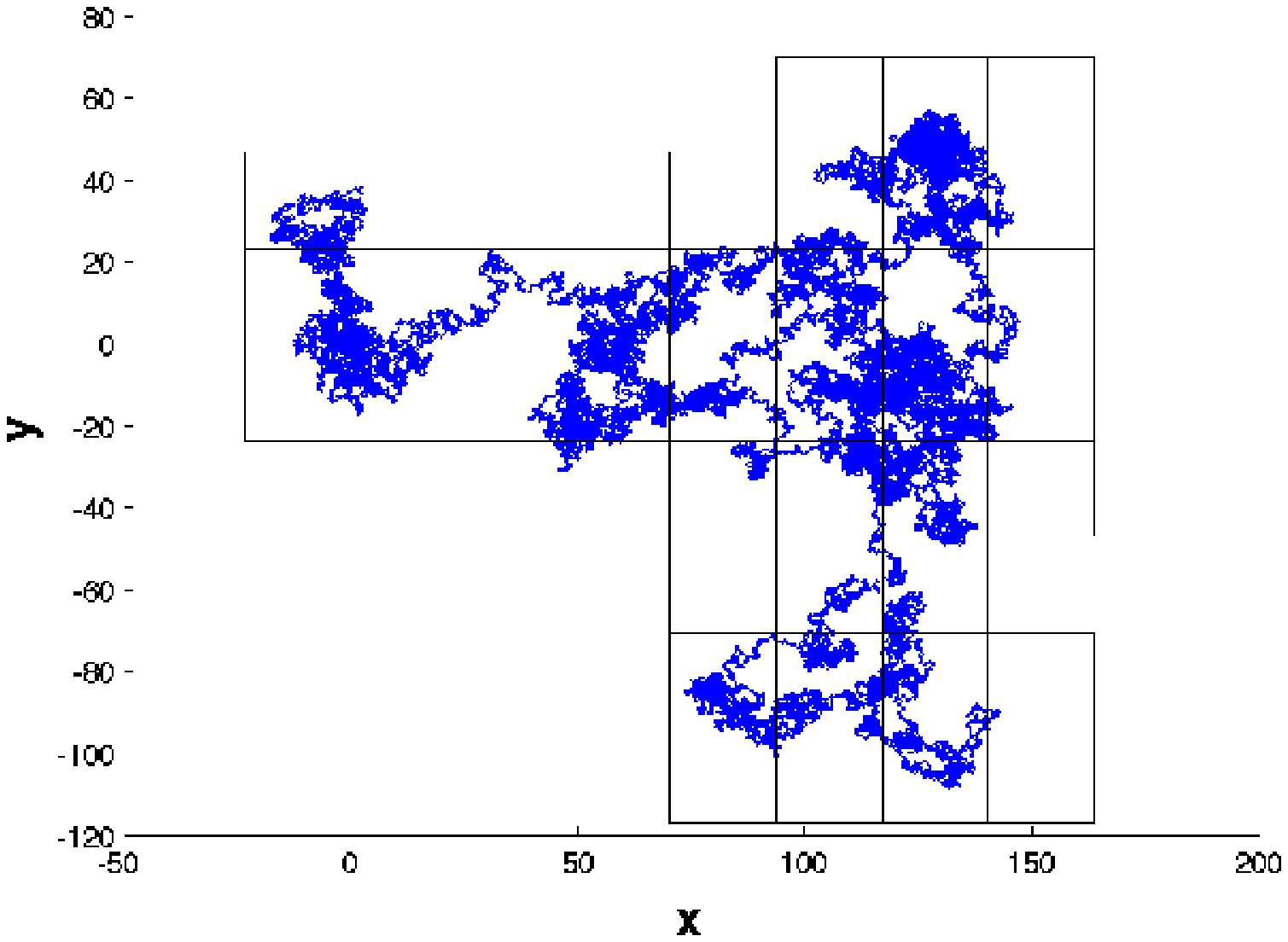,scale=0.3}} 
\caption{$x-y$ \Index{projection} of the path and the sub-cubes for different
values of $m$}
\label{easy_money1}
\end{center}
\end{figure}
\item \bf{Step 4:} \normalfont 
\emph{Generate uniform random variables $U_1$, \dots, $U_J$ in
  $[-1/2,\allowbreak 1/2]^d$ and estimate $V=|S_T|$ by
$$
\hat V=J^{-1}2^{-md}L^d\sum_{i=1}^I\sum_{j=1}^J\left(1-\prod_{n=0}^N1_{A_n(i,j)}\right)
$$
where $A_n(i,j)=\{y_i+2^{-m}LU_j\not\in K+X_{n\D t}\}$.}
\end{itemize}

The algorithm was tested in the case $d=3$.
We took $\ve=0.25$ and took $v$ to be the
\index{Taylor, B.!Taylor--Green vector field}Taylor--Green vector field in
the first two co-ordinate directions; thus
$$
v(x)=(-\sin x_1\cos x_2,\,\cos x_1\sin x_2,\,0)^T.
$$
We applied the algorithm to $X^{(r)}$, which has drift vector field $v^{(r)}(x)=rv(rx)$, for a range of choices of $r\in(0,\infty)$.
We took $K$ to be the Euclidean unit ball $B$ and computed $|S_T^B(X^{(r)})|$ for $T=10^{4}$, using parameter values $N=10^6$, $m=4$ and $J=10^{4}$. 
The numerical method used to solve (\ref{SDEE}) was taken from
\index{Pavliotis, G. A.}\index{Stuart, A. M.}\cite{GPASKZ1}.
The values $|S_T^B(X^{(r)})|/T$ were taken as estimates of the asymptotic volume growth rate $\g(0.25,v^{(r)},B)$.
These are displayed in Figure \ref{so_much}.
\begin{figure}[htb]  
\begin{center}
\epsfig{file=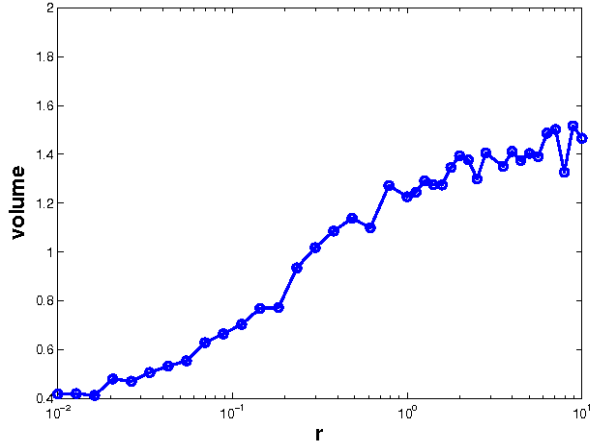,scale=0.5}
\caption{Growth rate of the sausage for different values of $r$}
\label{so_much}
\end{center}
\end{figure}

In Section \ref{DVE} we stated the following theoretical limit, deduced from Theorem \ref{AGRT}:
\begin{equation}\label{e:riders}
\lim_{r\to0}\g(\ve,v^{(r)},B)=\ve^2\capp(B)=2\pi\ve^2=0.3927.
\end{equation}
This is consistent with the computed values of $\g(0.25,v^{(r)},B)$ when $r$ is small.

It is known
\index{Majda, A. J.}\index{Kramer, P. R.}\cite{kramer} that $\bar a(\ve,v)$ has the form
\begin{equation*}\label{e:form_effective}
\bar a(\ve,v)=
\left(
\begin{array}{ccc}
\a  &  0 & 0   \\
0  &  \a & 0  \\
0  & 0  &  \ve^2/2  
\end{array}
\right)
\end{equation*}
for some $\a=\a(\ve,v)$, which can be computed using Monte Carlo
\index{simulation}simulations. 
In \cite{ZygThesis}, this was carried out for $\ve=0.25$ up to a final time $T=10^{4}$, using a time step 
$\Delta t=10^{-2}$, again using a numerical method from \cite{GPASKZ1} to solve (\ref{SDEE}). 
The value $\a(0.25,v)=0.0942$ was obtained as the sample average over $10^4$ realizations of $|X^1_T|^2/T$.
Using this value, we simulated $\bar X$ and used the volume algorithm to compute
$|S_T^B(\bar X)|/T$ as an approximation to $\capp_{\bar a(0.25,v)}(B)$, obtaining the value $1.4587$.
We showed theoretically that
\begin{equation*}\label{xe:riders}
\lim_{r\to\infty}\g(\ve,v^{(r)},B)=\capp_{\bar a(\ve,v)}(B).
\end{equation*}
The computed value for $\capp_{\bar a(0.25,v)}(B)$ is consistent with the
computed values of $\g(0.25,v^{(r)},B)$ for large
$r$.\index{Markov, A. A.!Markov chain Monte Carlo (MCMC)|)}

 \bibliography{hhnz}\label{refs}
 \bibliographystyle{cambridgeauthordate}

\end{document}